\documentclass[10pt,reqno]{amsart}

\usepackage[dvips]{graphicx}

\usepackage{charter}
\usepackage[a4paper,top=46mm,bottom=46mm,left=37mm,right=37mm]{geometry}

\usepackage{latexsym,amsfonts}
\usepackage{amsmath,amssymb,graphics,setspace}
\usepackage{amsthm}
\usepackage{mathrsfs}

\usepackage[T1]{fontenc}%

\makeatletter
\newcommand{\doublewedge}{\big@doubleop{\wedge}}
\newcommand{\big@doubleop}[1]{%
  \DOTSB\mathop{\mathpalette\big@doubleop@aux{#1}}\slimits@
}

\newcommand\big@doubleop@aux[2]{%
  \sbox\z@{$\m@th#1#2$}%
  \makebox[1.35\wd\z@][s]{$\m@th#1#2\hss#2$}%
}

\usepackage{marvosym}
\usepackage{paralist}
\usepackage{color}

\usepackage{marvosym}

\usepackage{pstricks,pst-text,pst-grad,pst-node,pst-3dplot,pstricks-add,pst-poly,pst-coil} 
\usepackage{pst-fun,pst-blur} 

\usepackage{picins,graphicx}
\usepackage{paralist}
\usepackage{color}

\usepackage[verbose]{wrapfig}

\usepackage{subfigure}
\renewcommand{\thesubfigure}{\thefigure.\arabic{subfigure}}
\makeatletter
\renewcommand{\p@subfigure}{}
\renewcommand{\@thesubfigure}{\thesubfigure:\hskip\subfiglabelskip}
\makeatother

\usepackage{hyperref}
\hypersetup{linktocpage=true,colorlinks=true,linkcolor=blue,citecolor=blue,pdfstartview={XYZ 1000 1000 1}}

\usepackage{floatflt,graphicx}

\usepackage{graphicx}

\newcommand{\cl}{\mbox{cl}}
\newcommand{\Int}{\mbox{int}}
\newcommand{\bdy}{\mbox{bdy}}

\newtheorem{example}{Example}
\newtheorem{remark}{Remark}

\newtheorem{lemma}{Lemma}
\newtheorem{theorem}{Theorem}

\usepackage{xcolor}
\definecolor{light}{gray}{0.80}

\begin{document}

\title[Descriptive Cellular Homology]{Descriptive Cellular Homology}

\author[M.Z. Ahmad]{M.Z. Ahmad$^{\alpha}$}
\email{ahmadmz@myumanitoba.ca}
\address{\llap{$^{\alpha}$\,}
Computational Intelligence Laboratory,
University of Manitoba, WPG, MB, R3T 5V6, Canada}

\author[J.F. Peters]{J.F. Peters$^{\beta}$}
\email{James.Peters3@umanitoba.ca}
\address{\llap{$^{\beta}$\,}
Computational Intelligence Laboratory,
University of Manitoba, WPG, MB, R3T 5V6, Canada and
Department of Mathematics, Faculty of Arts and Sciences, Ad\.{i}yaman University, 02040 Ad\.{i}yaman, Turkey}
\thanks{The research has been supported by the Natural Sciences \&
Engineering Research Council of Canada (NSERC) discovery grant 185986 
and Instituto Nazionale di Alta Matematica (INdAM) Francesco Severi, Gruppo Nazionale per le Strutture Algebriche, Geometriche e Loro Applicazioni grant 9 920160 000362, n.prot U 2016/000036.}

\subjclass[2010]{Primary 54E05 (Proximity); Secondary 68U05 (Computational Geometry)}

\date{}

\dedicatory{Dedicated to J.H.C. Whitehead and Som Naimpally}

\begin{abstract}
This article introduces descriptive cellular homology on cell complexes, which is an extension of J.H.C. Whitehead's CW topology.   A main result is that a descriptive cellular complex is a topology on fibres in a fibre bundle.    An application of two forms of cellular homology is given in terms of the persistence of shapes in CW spaces
\end{abstract}

\keywords{CW topology, Descriptive, Fibre bundle, Homology}

\maketitle
\tableofcontents

\section{Introduction}
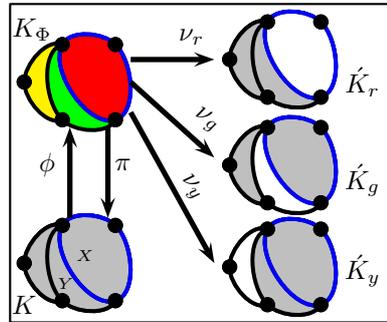
\begin{wrapfigure}{r}{0.45\textwidth}
\begin{minipage}[c]{0.46\textwidth}
\begin{pspicture}%
(0.3,-0.3)(5.4, 3.3)
\psframe[linecolor=black](0.3,-0.7)(5.37,3.55)
\psccurve[linewidth=1.5pt,fillstyle=solid,fillcolor=yellow](0.5,2.5)(1,2)(1,3)
\psccurve[linewidth=1.5pt,fillstyle=solid,fillcolor=green](1,2)(1,3)(1.7,2)
\psccurve[linewidth=1.5pt,fillstyle=solid,fillcolor=red](1,3)(1.7,3)(1.7,2)
\psccurve[linewidth=1.5pt, linecolor=blue,arrows=->](1,3)(1.7,3)(1.7,2)
\psccurve[linewidth=1.5pt,linecolor=black](0.5,0.1)(1,-0.4)(1,0.6)
\psccurve[linewidth=1.5pt,linecolor=black](1,-0.4)(1,0.6)(1.7,-0.4)
\psccurve[linewidth=1.5pt,linecolor=black](1,0.6)(1.7,0.6)(1.7,-0.4)
\psccurve[linewidth=1.5pt,linecolor=black, fillstyle=solid, fillcolor=lightgray](0.5,0.1)(1,-0.4)(1,0.6)
\psccurve[linewidth=1.5pt,linecolor=black,, fillstyle=solid, fillcolor=lightgray](1,-0.4)(1,0.6)(1.7,-0.4)
\psccurve[linewidth=1.5pt,linecolor=black,, fillstyle=solid, fillcolor=lightgray](1,0.6)(1.7,0.6)(1.7,-0.4)
\psccurve[linewidth=1.5pt, linecolor=blue,arrows=->](1,0.6)(1.7,0.6)(1.7,-0.4)
\psccurve[linewidth=1.5pt,linecolor=black](3.2,0.05)(3.7,-0.45)(3.7,0.55)
\psccurve[linewidth=1.5pt,linecolor=black](3.7,-0.45)(3.7,0.55)(4.4,-0.45)
\psccurve[linewidth=1.5pt,linecolor=black](3.7,0.55)(4.4,0.55)(4.4,-0.45)
\psccurve[linewidth=1.5pt,linecolor=black, fillstyle=solid, fillcolor=white](3.2,0.05)(3.7,-0.45)(3.7,0.55)
\psccurve[linewidth=1.5pt,linecolor=black, fillstyle=solid, fillcolor=lightgray](3.7,-0.45)(3.7,0.55)(4.4,-0.45)
\psccurve[linewidth=1.5pt,linecolor=black, fillstyle=solid, fillcolor=lightgray](3.7,0.55)(4.4,0.55)(4.4,-0.45)
\psccurve[linewidth=1.5pt, linecolor=blue,arrows=->](3.7,0.55)(4.4,0.55)(4.4,-0.45)
\psccurve[linewidth=1.5pt,linecolor=black](3.2,1.4)(3.7,0.9)(3.7,1.9)
\psccurve[linewidth=1.5pt,linecolor=black](3.7,0.9)(3.7,1.9)(4.4,0.9)
\psccurve[linewidth=1.5pt,linecolor=black](3.7,1.9)(4.4,1.9)(4.4,0.9)
\psccurve[linewidth=1.5pt,linecolor=black, fillstyle=solid, fillcolor=lightgray](3.2,1.4)(3.7,0.9)(3.7,1.9)
\psccurve[linewidth=1.5pt,linecolor=black, fillstyle=solid, fillcolor=white](3.7,0.9)(3.7,1.9)(4.4,0.9)
\psccurve[linewidth=1.5pt,linecolor=black, fillstyle=solid, fillcolor=lightgray](3.7,1.9)(4.4,1.9)(4.4,0.9)
\psccurve[linewidth=1.5pt, linecolor=blue,arrows=->](3.7,1.9)(4.4,1.9)(4.4,0.9)
\psccurve[linewidth=1.5pt,linecolor=black](3.2,2.8)(3.7,2.3)(3.7,3.3)
\psccurve[linewidth=1.5pt,linecolor=black](3.7,2.3)(3.7,3.3)(4.4,2.3)
\psccurve[linewidth=1.5pt,linecolor=black](3.7,3.3)(4.4,3.3)(4.4,2.3)
\psccurve[linewidth=1.5pt,linecolor=black, fillstyle=solid, fillcolor=lightgray](3.2,2.8)(3.7,2.3)(3.7,3.3)
\psccurve[linewidth=1.5pt,linecolor=black, fillstyle=solid, fillcolor=lightgray](3.7,2.3)(3.7,3.3)(4.4,2.3)
\psccurve[linewidth=1.5pt,linecolor=black, fillstyle=solid, fillcolor=white](3.7,3.3)(4.4,3.3)(4.4,2.3)
\psccurve[linewidth=1.5pt, linecolor=blue,arrows=->](3.7,3.3)(4.4,3.3)(4.4,2.3)
\psdots[dotstyle=o,dotsize=0.2,fillstyle=none,fillcolor=black](0.5,0.1)(1,0.6)(1,-0.4)(1.7,0.6)(1.7,-0.4)
\psdots[dotstyle=o,dotsize=0.2,fillstyle=solid, fillcolor=black](0.5,2.5)(1,3)(1,2)(1.7,3)(1.7,2)
\psdots[dotstyle=o,dotsize=0.2,fillstyle=solid, fillcolor=black](3.2,2.8)(3.7,3.3)(3.7,2.3)(4.4,3.3)(4.4,2.3)
\psdots[dotstyle=o,dotsize=0.2,fillstyle=solid, fillcolor=black](3.2,1.4)(3.7,1.9)(3.7,0.9)(4.4,1.9)(4.4,0.9)
\psdots[dotstyle=o,dotsize=0.2,fillstyle=solid, fillcolor=black](3.2,0.05)(3.7,0.55)(3.7,-0.45)(4.4,0.55)(4.4,-0.45)
\psline[linewidth=2pt,linecolor=black]{<-}(1.6,0.7)(1.6,1.9)
\psline[linewidth=2pt,linecolor=black]{->}(1.1,0.7)(1.1,1.9)
\psline[linewidth=2pt,linecolor=black]{->}(1.9,2.8)(3,2.8)
\psline[linewidth=2pt,linecolor=black]{->}(1.9,2.5)(3,1.5)
\psline[linewidth=2pt,linecolor=black]{->}(1.9,2.1)(3,0.1)
\rput(0.5,-0.45){$K$}
\rput(0.6,3.2){$K_{\Phi}$}
\rput(0.8,1.4){$\phi$}
\rput(1.8,1.4){$\pi$}
\rput(2.7,3.1){$\nu_{r}$}
\rput{-30}(2.9,2){$\nu_{g}$}
\rput{-50}(2.7,1.1){$\nu_{y}$}
\rput(4.95,2.5){$\acute{K}_{r}$}
\rput(4.95,1.2){$\acute{K}_{g}$}
\rput(4.95,0){$\acute{K}_{y}$}
\rput(1.3,0.2){\tiny $X$}
\rput(1.05,-0.17){\tiny $Y$}
\end{pspicture}
\caption{Descriptive Homology Space}
\label{fig:deshom}
\end{minipage}
\end{wrapfigure}

Homology separates cycles (connected paths) from boundaries of holes in shapes in topological spaces.  Descriptive cellular homology provides a means of characterizing and comparing cycles and boundaries in terms of descriptions of cell complexes. 
A \emph{cell complex} $K$ on a space $X$ is a finite collection of subsets of $X$ called cells such as 0-cells (vertices) and 1-cells (open arcs).     
A descriptive cellular homology is an extension of Whitehead Closure-finite Weak (CW) topology~\cite{Whitehead1949BAMS-CWtopology}, which is a form of descriptive topology introduced in~\cite{Naimpally2013}.  A straightforward application of these two forms of homology is the study persistence of shapes in CW spaces.

\section{\textcolor{blue}{Preliminaries}}
A \emph{Hausdorff space} is a nonempty set $K$ in which every pair of distinct elements in $K$ is contained in disjoint open set.   A decomposition of $K$ is a partition of $K$ into subsets called \emph{cells}.   Let  $\cl A, bdy A$ denote the closure and boundary of a cell $A$, respectively.   The interior of a cell $A$ (denoted by $\Int A$) is defined by $\Int A = \cl A- \bdy A$.   A complex with $n$ cells in $K$ is denoted by $\sigma^n$.   The closure of $\sigma^n$ (denoted by $\cl \sigma^n$) is the image of an n-simplex $\sigma^n$ in a continuous homomorphic map $f:\sigma^n\longrightarrow \cl \sigma^n$.  A Hausdorff space $K$ with a cell decomposition is called a \emph{CW complex}~\cite{Whitehead1949BAMS-CWtopology} (briefly, \emph{complex}) in which the following conditions hold.
\begin{enumerate}
\item For each cell $\sigma\in K$, $\cl A$ intersects a finite number of other cells (\emph{Closure finiteness}).
\item A cell $A\in K$ is closed, provided $A\cap \cl \sigma^n\neq \emptyset$ is also closed  (\emph{Weak topology}).
\end{enumerate}
The \emph{$n$-skeleton} $K^{n}$ is the union of all $\sigma^j \in K$, such that $j \leq n$. 
A \emph{fibre bundle} $(E,B, \pi, F)$ is a structure, where $E$ is the \emph{total space}, $B$ is the \emph{base space}, $\pi:E \rightarrow B$ is a continuous surjection on $E$ onto $B$ and $F\subset E$ is called the \emph{fibre}. A \emph{region based probe function} $\phi:2^K \rightarrow \mathbb{R}^n$ attaches description to sets.
A \emph{descriptive cell complex} $K_{\Phi}$ is a fibre bundle, $(K_{\Phi},K, \pi, \phi(U))$ on a complex $\sigma^n\in K$, and $U \subset K$.

\begin{example} {Sample Descriptive Cell Complex}.\\
In Fig.~\ref{fig:deshom}, $K$ is a CW complex.  A region-based probe function $\phi:2^K \rightarrow \mathbb{R}$ is used to attach description to each of the cells in $K$ (denoted by $K_{\Phi}$). Probes define a signature for a cell complex~\cite{Peters2017AMSJsignature}.  In $K_{\Phi}$, every cell has a description defined by an appropriate function, as opposed to $K$, where the description is assumed to be constant. The map $\pi:K_{\Phi} \rightarrow K$, is a continuous surjective projection map. Thus, the descriptive  CW complex is a fibre bundle $(K_{\Phi},K, \pi, \phi(U))$, where $U \in K$. The region-based probe is the section of the bundle. Moreover, we illustrate a family of maps $\{\nu_\alpha\;s.t.\; \alpha \in img\,\phi(K)\}$, each of which maps from $K_{\Phi}$ to a CW complex $\hat{K}_\alpha$.
A cycle of $\sigma^1$ is illustrated in blue color. The idea is that $p$-cycle is a chain or a formal sum of $\sigma^p$ which starts and ends at the same point.
Each region of constant description in $K_{\Phi}$ is a descriptive hole and the family of maps $\{\nu_\alpha\}_\alpha$ maps $K_{\Phi}$ in such a way that a particular descriptive hole is projected to a hole (a void) in a abstract CW complex $\acute{K}_\alpha$. 
\qquad \textcolor{blue}{\Squaresteel}
\end{example}

A descriptive cellular complex is represented as $F \rightarrow E \overset{\pi}{\longrightarrow} B$.
The fibre bundle must satisfy the local trivialization condition, \em i.e., in a small neighborhood $U\subset B$, such that $\pi^{-1}(U)$ is homeomorphic to $U \times F$. 
This means that the diagram in Fig.~\ref{fig:triv}, commutes, i.e. each directed path with same endpoints lead to same result. 
Also from this diagram, the function $\phi$ is the \emph{section}, i.e. a continuous map from the base space to the total space, of the descriptive cellular complex $(X_{\Phi},X,pi,\phi(U))$.

\begin{wrapfigure}{l}{0.4\textwidth}
\begin{minipage}[c]{.41\textwidth}
\centering
\includegraphics[width=50mm]{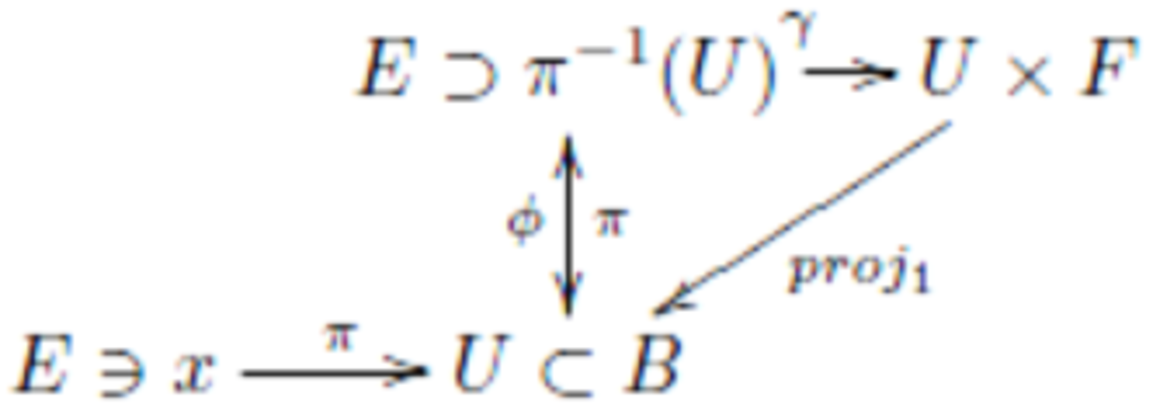}
\caption{Local Trivialization Condition}
\label{fig:triv}
\end{minipage}
\end{wrapfigure}

Algebraic structures are defined over CW approximations of topological spaces.   A \emph{$p$-chain} is a formal sum of $p$-cells($\sigma^p$) in $X$, $c_p=\sum a_i \sigma^p$. 
For computational simplicity, $a_i \in \{0,1\}$. 
Under component-wise addition, the set of $p$-chains would form an abelian group $(C_p^{\sigma},+)$. These are to be distinguished from the simplicial and singular chain complexes $(C_p,+)$. 
\emph{Boundary map} is defined as, $\partial^{\sigma}:\sigma^p_\alpha \longmapsto \sum_\beta d_{\alpha \beta}\sigma_{\beta}^{p-1}$. 
Here $\sigma^p_\alpha \in \{\sigma^p_\alpha\}_\alpha$ is the set of all $\sigma^p \in K$, and $\sigma^p_\beta \in \{\sigma^p_\beta\}_\beta$, is the set of all $\sigma^p \in K$. $d_{\alpha\beta}$ is the degree of the map $f:S_\alpha^{p-1} \rightarrow K^{p-1} \rightarrow S_\beta^{p-1}$. 
Observe that $S_\alpha^{p-1} \rightarrow K^{n-1}$ is the attachment map of $\sigma_\alpha^{p}$ and $K^{p-1} \rightarrow S_\beta^{p-1}$ is the quotient map that collapses $\{K^{p-1}-\sigma_{\beta}\}$ to a point.
The map $f:S_\alpha^{p-1} \rightarrow S_\beta^{p-1}$ induces a homomorphism between infinite cyclic groups as
$f_{\star}:H_n(S_\alpha^{p-1})\rightarrow H_n(S_\beta^{p-1})$ and hence should be of the form $f_{\star}(\alpha)=d\alpha$. The constant $d$ is called the \emph{degree} of the map $f$\cite[\S~ 2.2]{Hatcher2002CUPalgebraicTopology}.
The boundary map is a homomorphism between chain groups, $\partial^{\sigma}:C_{p+1}^{\sigma}\rightarrow C_p^{\sigma}$. 
The group of $p$-cycles is defined as the kernel of boundary homomorphism, $Z_p^{\sigma}=ker\,\partial^{\sigma}$. 
The group of $p$-boundaries are the defined as, $B_p^{\sigma}=img\, \partial^{\sigma}$. 
The \emph{$p^{th}$ homology group} is the quotient group,$H_p^{\sigma}=Z_p^{\sigma}/B_p^{\sigma}$. 
The rank of $H_p^{\sigma}$ is the $p^{th}$ \emph{Betti number}, $\beta_p=rank \,H_p^{\sigma}$.  

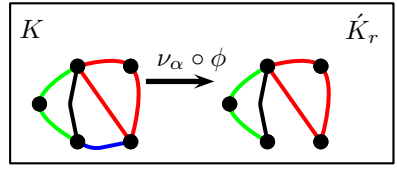
\begin{wrapfigure}[10]{R}{0.45\textwidth}
\begin{minipage}[c]{0.45\textwidth}
\begin{pspicture}(0.3,-0.3)(5.4, 1.6)
\psframe[linecolor=black](0.3,-0.6)(5.37,1.55)
\pscurve[linewidth=1.5pt, linecolor=green](1.2,-0.3)(0.7,0.2)(1.2,0.7)
\psline[linearc=0.2,linewidth=1.5pt](1.2,-0.3)(1.1,0.2)(1.2,0.7)
\psline[linearc=0.35, linewidth=1.5pt, linecolor=blue](1.2,-0.3)(1.4,-0.4)(1.9,-0.3)
\pscurve[linewidth=1.5pt, linecolor=red](1.2,0.7)(1.9,0.7)(1.9,-0.3)
\psline[linearc=0.2,linecolor=red,linewidth=1.5pt](1.2,0.7)(1.9,-0.3)
\pscurve[linewidth=1.5pt, linecolor=green](3.7,-0.3)(3.2,0.2)(3.7,0.7)
\psline[linearc=0.2,linewidth=1.5pt](3.7,-0.3)(3.6,0.2)(3.7,0.7)
\pscurve[linewidth=1.5pt, linecolor=red](3.7,0.7)(4.4,0.7)(4.4,-0.3)
\psline[linearc=0.2,linewidth=1.5pt, linecolor=red](3.7,0.7)(4.4,-0.3)
\psdots[dotstyle=o,dotsize=0.2,fillstyle=solid, fillcolor=black](0.7,0.2)(1.2,0.7)(1.2,-0.3)(1.9,0.7)(1.9,-0.3)
\psdots[dotstyle=o,dotsize=0.2,fillstyle=solid, fillcolor=black](3.2,0.2)(3.7,0.7)(3.7,-0.3)(4.4,0.7)(4.4,-0.3)
\psline[linewidth=2pt,linecolor=black]{->}(2.1,0.5)(3,0.5)
\rput(0.6,1.2){$K$}
\rput(2.7,0.8){$\nu_{\alpha}\circ \phi$}
\rput(4.95,1.2){$\acute{K}_{r}$}
\end{pspicture}
\caption{Descriptive relative cycles}
\label{fig:lem1}
\end{minipage}
\end{wrapfigure}

Next, we establish the connection between singular homology and cellular homology. The chain complexes in cellular homology, $(C_p^{\sigma},+)$ are the relative homology groups $H_n(K^{n},K^{n-1}) \equiv H_n(K^{n}/K^{n-1})$. The underlying chain groups are $(C_p(K^{n})/C_p(K^{n-1}),+)$. The elements in $H_n(K^{n},K^{n-1})$ are represented by \emph{relative chains}, $c_r \in C_n(K^{n})$ such that $\partial c_r \in C_n(K^{n-1})$. The \emph{relative boundaries} are defined as, $c_r =\partial \hat{c}_r +\gamma$ such that $\partial c_r \in C_n(K^{n})$, $\hat{c}_r \in C_{n+1}(K^{n+1})$ and $\gamma \in C_{n-1}{K^{n-1}}  $. The relative boundaries are trivial elements in $H_n(K^{n},K^{n-1})$. Thus, $H_n(K^{n},K^{n-1})$ has $\sigma^{n}\in K$ as its basis. For the details, see A. Hatcher~\cite{Hatcher2002CUPalgebraicTopology}.

\section{Main Results}\label{sec:main_results}

This section introduces the main results for a descriptive cellular homology. Let $\nu_{\alpha}:K_{\Phi} \rightarrow \acute{K}_\alpha$ be a mapping on  a descriptive cellular complex($K_{\Phi}$) into abstract cellular complex($\acute{K}_\alpha$). 
Here, $\acute{K}_\alpha=\{K \setminus \sigma^2 \, s.t. \, \sigma^2 \in K \, and \,  |\phi(\sigma^2)-\alpha| \leq \delta \}$ , and $\delta$ is an arbitrary constant. 
From this definition, $\acute{K}_\alpha \subset K_{\Phi}$, where $\alpha \in img\,\phi$, and $X_\phi=\bigcup_{\alpha} \acute{K}_\alpha$. We define homology groups, $H_{0}^{\alpha}, H_{1}^{\alpha}$ on each of these complexes $\acute{K}_\alpha$. 
These groups require cellular chain groups,$H_n(\hat{K}_{\alpha}^{n},\hat{K}_{\alpha}^{n-1})$ and boundary homomorphisms, $\partial^\sigma$. 
\begin{lemma}\label{lem:subgr}
Let $(K,K_{\Phi},\pi)$ be a descriptive cellular complex, $\{\nu_{\alpha i}:K_{\Phi} \rightarrow \acute{K}_{\alpha i} \, s.t. $ $ \alpha i \in img\,\phi \}$ be a class of maps. 
Then, $H_n(\acute{K}_{\alpha}^{n},\acute{K}_{\alpha}^{n-1})$ is a subgroup of $H_n(K^{n},K^{n-1})$.
\end{lemma} 
\begin{proof}
As $\acute{K}_{\alpha i}\subseteq K$, all $\acute{\sigma}^p \in \acute{K}_{\alpha i}$ are included in $K$. Hence, $H_n(\acute{K}_{\alpha}^{n},\acute{K}_{\alpha}^{n-1})$, with basis $\acute{\sigma^p} \in \acute{K}_{\alpha i}$ is a subset of $H_n(K^{n},K^{n-1})$, a basis $\sigma^p \in K$. Equipped with a chain addition, we can identify the identity element. The identity element in $H_n(\acute{K}_{\alpha}^{n},\acute{K}_{\alpha}^{n-1})$ can be constructed by setting the weight of each $\acute{\sigma}^p \in \acute{K}_{\alpha i}$  equal to $0$. This identity can be seen as the result of restriction of the identity of $H_n(K^{n},K^{n-1})$ to only those $\sigma \in K$ which are also $\acute{\sigma}^n \in \acute{K}_{\alpha i}$. The inverse of each element $c=\sum_i a_i\acute{\sigma}^n_i \in H_n(\acute{K}_{\alpha}^{n},\acute{K}_{\alpha}^{n-1})$ is $\acute{c}=\sum_i -a_i\acute{\sigma}^n_i$. The inverse of each element is by definition in $H_n(\acute{K}_{\alpha}^{n},\acute{K}_{\alpha}^{n-1})$ as it is a formal sum of $\acute{\sigma}^n$. Assume two chains $c=\sum_j a_j \acute{\sigma}^n$ and $\acute{c}=\sum_j b_j \acute{\sigma}^n$, then $c+\acute{c}=\sum_j (a_j+b_j)\acute{\sigma}^n$, By definition $c+\acute{c} \in H_n(\acute{K}_{\alpha}^{n},\acute{K}_{\alpha}^{n-1})$ as it is a formal sum of $\acute{\sigma}^n \in \acute{K}_{\alpha i}$. Thus, $H_n(\acute{K}_{\alpha}^{n},\acute{K}_{\alpha}^{n-1})$ is closed under chain addition. It can be seen from the definition of the cellular boundary map that coefficients of $\acute{\sigma}^n$ are integers. Thus associativity of chain addition in $H_n(\acute{K}_{\alpha}^{n},\acute{K}_{\alpha}^{n-1})$ follows from the associativity of integers, $\mathbb{Z}$. As $H_n(\acute{K}_{\alpha}^{n},\acute{K}_{\alpha}^{n-1}) \subseteq H_n(K^{n},K^{n-1})$ and $H_n(\acute{K}_{\alpha}^{n},\acute{K}_{\alpha}^{n-1})$ is it self a group. Hence, $H_n(\acute{K}_{\alpha}^{n},\acute{K}_{\alpha}^{n-1})$ is a subgroup of $H_n(K^{n},K^{n-1})$.   
\end{proof}

\begin{example}
Let us look at an illustration in Fig.~\ref{fig:lem1} to clarify Lemma~\ref{lem:subgr}. The cellular complex $K$ is a union of $\sigma^1$. 
The probe function is chosen as the curvature. 
Thus, using the function $\nu_{\alpha} \circ \phi$ we can construct $\acute{K}_{\alpha}$, by removing the blue arc from $K$. 
We can see that the cycle in red, $cyc \acute{A} \in \acute{K}_{\alpha}$ is also a cycle in $cyc A \in K$. 
The same is true for the arc in black i.e. $arc \acute{B}\in \acute{K}_{\alpha} \simeq arc B \in K$, and for the arc in green i.e. $arc \acute{C}\in \acute{K}_{\alpha} \simeq arc C\in K$. Thus all the arcs and the cycles in $\acute{K}_{\alpha}$ are in $K$, as all the $\sigma^1 \in \acute{K}_{\alpha} \simeq \sigma^1 \in K$. 
Since $H_1(\acute{K}_\alpha^1,\acute{K}_\alpha^0)$ is a formal sum of $\sigma^1 \in \acute{K}_{\alpha}$ and $H_1(K^1,K^0)$ is a formal sum of $\sigma^1 \in K$, $H_1(\acute{K}_\alpha^1,\acute{K}_\alpha^0) \subseteq H_1(K^1,K^0)$. 
We just have to proof that $H_1(\acute{K}_\alpha^1,\acute{K}_\alpha^0)$ is itself a group as in the proof of Lemma~\ref{lem:subgr}. 
Moreover, we can extend the argument to $H_n(\acute{K}_\alpha^1,\acute{K}_\alpha^0)$.
\qquad \textcolor{blue}{\Squaresteel}
\end{example}

Using Lemma~\ref{lem:subgr}, we show that the Fundamental lemma of homology is satisfied for the chain groups $H_n(\hat{K}_{\alpha}^{n},\hat{K}_{\alpha}^{n-1})$.
\begin{lemma}\label{lem:fund}
Let us consider $H_n(\hat{K}_{\alpha}^{n},\hat{K}_{\alpha}^{n-1})$ to be chain groups with boundary homomorphisms associated as follows:
\begin{align*}
\cdots \overset{\partial_{n+2}^{\sigma}}{\rightarrow} H_{n+1}(\hat{K}_{\alpha}^{n+1},\hat{K}_{\alpha}^{n}) \overset{\partial_{n+1}^{\sigma}}{\rightarrow} H_n(\hat{K}_{\alpha}^{n},\hat{K}_{\alpha}^{n-1})\overset{\partial_n}
{\rightarrow} H_{n-1}(\hat{K}_{\alpha}^{n-1},\hat{K}_{\alpha}^{n-2}) \cdots.
\end{align*}
Then $\partial_{n+1}^\sigma \circ \partial_{n}^\sigma=0$.
\end{lemma}
\begin{proof}
The composition $\partial_{n+1}^\sigma \circ \partial_{n}^\sigma=0$ holds for $H_n(K^n,K^{n-1})$ as a result of Lemma $2.34$~\cite{Hatcher2002CUPalgebraicTopology}. The proof via diagram chasing, is detailed in \cite[\S~2.2, p.139]{Hatcher2002CUPalgebraicTopology}. 

Using this proof as the basis and the Lemma~\ref{lem:subgr} stating that $H_n(\hat{K}_{\alpha}^{n},\hat{K}_{\alpha}^{n-1})$ is a subgroup of $H_n(K^n,K^{n-1})$, it can be concluded that $\partial_{n+1}^\sigma \circ \partial_{n}^\sigma=0$ holds for $H_n(\hat{K}_{\alpha}^{n},\hat{K}_{\alpha}^{n-1})$.    
\end{proof}

\begin{theorem}
$img \partial_{n+1}^{\sigma} \subset ker \partial_{n}^{\sigma}$.
\end{theorem}
\begin{proof}
From Lemma~\ref{lem:fund}, we can conclude that $img \partial_{n+1}^{\sigma} \subset ker \partial_{n}^{\sigma}$. Thus, the sequence  $\cdots \overset{\partial_{n+2}^{\sigma}}{\rightarrow} H_{n+1}(\hat{K}_{\alpha}^{n+1},\hat{K}_{\alpha}^{n}) \overset{\partial_{n+1}^{\sigma}}{\rightarrow} H_n(\hat{K}_{\alpha}^{n},\hat{K}_{\alpha}^{n-1})\overset{\partial_n}{\rightarrow} H_{n-1}(\hat{K}_{\alpha}^{n-1},\hat{K}_{\alpha}^{n-2}) \cdots$ is an \emph{exact sequence}.
\end{proof}

\begin{theorem}
The condition $|\phi(\sigma^2)-\alpha|\leq \delta$ is equivalent to $\phi(\sigma^2) \in \tau_{std}^{\phi}$.
\end{theorem}
\begin{proof}
For a space $X$, we define a topology $(img \phi(X),\tau_{std}^{\phi})$. Any set  $U \in \tau_{std}^{\phi}$ is defined as: for all $p \in X$, there exists an arbitrary positive real number $\delta$ such that $B_{\phi(p)}^\delta \subseteq U$. Here $B_x^r$ is a ball of radius $r$ and centered on a point $x$. Hence from this definition, we can see that $|\phi(\sigma^2)-\alpha|\leq \delta \Rightarrow \phi(\sigma^2) \in \tau_{std}^{\phi}$. Moreover, the statement $\phi(\sigma^2) \in \tau_{std}^{\phi} \Rightarrow |\phi(\sigma^2)-\alpha|\leq \delta$, also follows from the definition of $\tau_{std}^{\phi}$. From this argument it follows that, $|\phi(\sigma^2)-\alpha|\leq \delta$ is equivalent to $\phi(\sigma^2) \in \tau_{std}^{\phi}$.
\end{proof}

Descriptive homology gives a local view of a space, since it is associated with subgroups $H_n(\hat{K}_{\alpha}^{n},\hat{K}_{\alpha}^{n-1})$ of $H_n(K^{n},K^{n-1})$. 
Next, we establish a relationship between $H_p(K)$ and $H_p^\alpha(K)$. To do this, we assume a standard topology $(img\, \phi(X),\tau_{std}^{\phi})$, and define $\acute{K}_\alpha=\{K \setminus \sigma^2 \, s.t. \, \sigma^2 \in K \, and \,  |\phi(\sigma^2)-\alpha| \leq \delta \}$. The the $p^{th}$ homology group is denoted by $H_p^V$, where $V\in\tau_{std}^{\phi}$.
\begin{theorem}\label{thm:clasrel}
Let $(X_phi,X,\pi, \phi(U))$ be a descriptive cellular complex where $U\in X$. $(img \phi(X),\tau_{std}^{\phi})$ is the topology defined on the fibres.Then,
\begin{align*}
H_p^V \equiv H_p(X)\; when \,V=\bigcup \tau_{std}^{\phi}
\end{align*}
\end{theorem}
\begin{proof}
$H_p^V$ is defined as the cellular homology group associated with $\acute{K}_\alpha=\{X \setminus \sigma^p \, s.t. \, \sigma^p \in K \, and \,  |\phi(\sigma^p)-\alpha| \leq \delta \}$. If $V=\bigcup \tau_{std}^{\phi}= img\, \phi(K)$ Then, $\acute{K}_\alpha=X$ as all the two simplices regardless of their description would be included.Thus, by definition it would equal the classical cellular homology group $H_p$ of space $K$.
\end{proof}

\begin{remark}
Each fibre bundle has a local trivialization as illustrated in Fig.~\ref{fig:triv}, and the base $K$ in the case of a cellular complex has intersecting subsets. This raises the question of how the fibre $\phi(U)$ transitions between such intersecting sets. This can be done by associating with $(K_\phi,K,\pi,\phi(U))$ a topological group $G$, which acts continuously on the fiber, $\phi(U)$ from the left.
That is, for $e\in G$, the identity element, $ex=x$, where $x\in \phi(U)$. The notion of continuity requires the group $G$ to be a topological group.
Let us formalize the action of group $G$ on the fibre.
\end{remark}

\begin{theorem}\label{thm:gauge group}
Let $(E,B,\pi,F)$ be the fibre bundle, where $E$ is the total space, $B$ is the base space, $\pi$ is a continuous surjection and $F$ is the fibre.
Let $(U_i, \phi_i)$ and $(U_j, \phi_j)$ be two intersecting sets, $U_i \cap U_j \neq \emptyset$, in $B$ with their sections $\phi_i, \phi_j$. Then the following holds:
\begin{align*}
&\phi_i \circ \phi_j^{-1}: (U_i \cap U_j) \times F \rightarrow (U_i \cap U_j) \times F\\
&\phi_i \circ \phi_j^{-1}(x,f) \longmapsto (x,t_{ij}(x)f) \; s.t. \; t_{ij}:U_i \cap U_j \rightarrow G. 
\end{align*}
\end{theorem}
\begin{proof}
The fibre bundle satisfies local trivialization condition.  Consequently, $\phi_i:U_i \rightarrow U_i \times F$ and $\phi_j:U_j \rightarrow U_j \times F$. 
Since, we assume $U_j \cap U_i \neq \emptyset$, there are two different trivializations for the region $U_i \cap U_j$. 
The map $\phi_j^{-1}:U_j \times F \rightarrow U_j$. 
Since a region $U_i \cap U_j$  has two trivializations in the fibre bundle, we can shift between the two. 
Thus, we can define a composition map, $\phi_i \circ \phi_j^{-1}: (U_i \cap U_j) \times F \rightarrow (U_i \cap U_j) \times F$. 
Such a map can be defined as $\phi_i \circ \phi_j^{-1}(x,f) \longmapsto (x,t_{ij}(x)f)$. In this case $t_{ij}\in G$ and $G$ is the structure group or the gauge group defined in \cite{FelixOprea2009PAMSgaugeGroups}.
\end{proof}

\begin{figure}
\vspace{-5mm}
\begin{subfigure}[Changing temperature of regions in space ]
{
\begin{pspicture}(0,-0.95)(5.4, 3.05)
\psframe[linecolor=black](0.05,-0.95)(5.37,3.05)
\psccurve[linewidth=1.5pt, fillstyle=solid, fillcolor=gray!20](0.5,-0.3)(0.5,0.5)(1.3,-0.3)
\psccurve[linewidth=1.5pt,fillstyle=solid, fillcolor=gray](1.3,-0.3)(1.3,0.5)(0.5,0.5)
\psccurve[linewidth=1.5pt, fillstyle=solid, fillcolor=green](0.5,1.7)(0.5,2.5)(1.3,1.7)
\psccurve[linewidth=1.5pt,fillstyle=solid, fillcolor=red](1.3,1.7)(1.3,2.5)(0.5,2.5)
\psccurve[linewidth=1.5pt, fillstyle=solid, fillcolor=gray!20](2.3,-0.3)(2.3,0.5)(3.1,-0.3)
\psccurve[linewidth=1.5pt,fillstyle=solid, fillcolor=gray](3.1,-0.3)(3.1,0.5)(2.3,0.5)
\psccurve[linewidth=1.5pt, fillstyle=solid, fillcolor=green](2.3,1.7)(2.3,2.5)(3.1,1.7)
\psccurve[linewidth=1.5pt,fillstyle=solid, fillcolor=yellow](3.1,1.7)(3.1,2.5)(2.3,2.5)
\psccurve[linewidth=1.5pt, fillstyle=solid, fillcolor=gray!20](4.1,-0.3)(4.1,0.5)(4.9,-0.3)
\psccurve[linewidth=1.5pt,fillstyle=solid, fillcolor=gray](4.9,-0.3)(4.9,0.5)(4.1,0.5)
\psccurve[linewidth=1.5pt, fillstyle=solid, fillcolor=green](4.1,1.7)(4.1,2.5)(4.9,1.7)
\psccurve[linewidth=1.5pt,fillstyle=solid, fillcolor=green](4.9,1.7)(4.9,2.5)(4.1,2.5)
\psdots[dotstyle=o,dotsize=0.2,fillstyle=solid, fillcolor=black](0.5,-0.3)(0.5,0.5)(1.3,-0.3)(1.3,0.5)
\psdots[dotstyle=o,dotsize=0.2,fillstyle=solid, fillcolor=black](2.3,-0.3)(2.3,0.5)(3.1,-0.3)(3.1,0.5)
\psdots[dotstyle=o,dotsize=0.2,fillstyle=solid, fillcolor=black](4.1,-0.3)(4.1,0.5)(4.9,-0.3)(4.9,0.5)
\psdots[dotstyle=o,dotsize=0.2,fillstyle=solid, fillcolor=black](0.5,1.7)(0.5,2.5)(1.3,1.7)(1.3,2.5)
\psdots[dotstyle=o,dotsize=0.2,fillstyle=solid, fillcolor=black](2.3,1.7)(2.3,2.5)(3.1,1.7)(3.1,2.5)
\psdots[dotstyle=o,dotsize=0.2,fillstyle=solid, fillcolor=black](4.1,1.7)(4.1,2.5)(4.9,1.7)(4.9,2.5)
\psline[linewidth=1.5pt]{->}(1,1.45)(1,0.65)
\psline[linewidth=1.5pt]{->}(0.6,0.7)(0.6,1.5)
\psline[linewidth=1.5pt]{->}(2.8,1.45)(2.8,0.65)
\psline[linewidth=1.5pt]{->}(2.4,0.7)(2.4,1.5)
\psline[linewidth=1.5pt]{->}(4.6,1.45)(4.6,0.65)
\psline[linewidth=1.5pt]{->}(4.2,0.7)(4.2,1.5)
\rput(1.2,1){$\pi$}
\rput(0.4,1){$\phi$}
\rput(3,1){$\pi$}
\rput(2.2,1){$\phi$}
\psline[linewidth=1.5pt]{->}(1.55,2)(2.15,2)
\psline[linewidth=1.5pt]{->}(3.35,2)(3.95,2)
\psline[linewidth=1.5pt]{->}(0.5,-0.75)(4.3,-0.75)
\rput(4.8,1){$\pi$}
\rput(4,1){$\phi$}
\rput(0.9,2.82){\small $K_\Phi(\theta_i)$}
\rput(2.7,2.82){\small $K_\Phi(\theta_j)$}
\rput(4.5,2.82){\small $K_\Phi(\theta_k)$}
\rput(1.8,2.3){\small $\psi_{ji}$}
\rput(3.6,2.3){\small $\psi_{kj}$}
\rput(4.8,-0.7){\small time}
\end{pspicture}
\label{subfig:temp_persist}
}
\end{subfigure}
\begin{subfigure}[Changing area of regions in space ]
{
\begin{pspicture}(0,-0.95)(5.4, 3.05)
\psframe[linecolor=black](0.05,-0.95)(5.37,3.05)
\psccurve[linewidth=1.5pt, fillstyle=solid, fillcolor=gray!20](0.5,-0.3)(0.5,0.5)(1.3,-0.3)
\psccurve[linewidth=1.5pt,fillstyle=solid, fillcolor=gray](1.3,-0.3)(1.3,0.5)(0.5,0.5)
\psccurve[linewidth=1.5pt, fillstyle=solid, fillcolor=green](0.5,1.7)(0.5,2.5)(1.3,1.7)
\psccurve[linewidth=1.5pt,fillstyle=solid, fillcolor=red](1.3,1.7)(1.3,2.5)(0.5,2.5)
\psccurve[linewidth=1.5pt, fillstyle=solid, fillcolor=gray!20](2.3,-0.3)(2.3,0.5)(3.1,-0.3)
\psccurve[linewidth=1.5pt,fillstyle=solid, fillcolor=gray](3.1,-0.3)(3.1,0.5)(2.3,0.5)
\psccurve[linewidth=1.5pt, fillstyle=solid, fillcolor=green](2.3,1.7)(2.5,2.5)(3.1,1.7)
\psccurve[linewidth=1.5pt,fillstyle=solid, fillcolor=red](3.1,1.7)(3.1,2.5)(2.5,2.5)
\psccurve[linewidth=1.5pt, fillstyle=solid, fillcolor=gray!20](4.1,-0.3)(4.1,0.5)(4.9,-0.3)
\psccurve[linewidth=1.5pt,fillstyle=solid, fillcolor=gray](4.9,-0.3)(4.9,0.5)(4.1,0.5)
\psccurve[linewidth=1.5pt, fillstyle=solid, fillcolor=green](4.1,1.7)(4.3,2.5)(5.1,1.9)
\psccurve[linewidth=1.5pt,fillstyle=solid, fillcolor=red](5.1,1.9)(4.9,2.5)(4.3,2.5)
\psdots[dotstyle=o,dotsize=0.2,fillstyle=solid, fillcolor=black](0.5,-0.3)(0.5,0.5)(1.3,-0.3)(1.3,0.5)
\psdots[dotstyle=o,dotsize=0.2,fillstyle=solid, fillcolor=black](2.3,-0.3)(2.3,0.5)(3.1,-0.3)(3.1,0.5)
\psdots[dotstyle=o,dotsize=0.2,fillstyle=solid, fillcolor=black](4.1,-0.3)(4.1,0.5)(4.9,-0.3)(4.9,0.5)
\psdots[dotstyle=o,dotsize=0.2,fillstyle=solid, fillcolor=black](0.5,1.7)(0.5,2.5)(1.3,1.7)(1.3,2.5)
\psdots[dotstyle=o,dotsize=0.2,fillstyle=solid, fillcolor=black](2.3,1.7)(2.5,2.5)(3.1,1.7)(3.1,2.5)
\psdots[dotstyle=o,dotsize=0.2,fillstyle=solid, fillcolor=black](4.1,1.7)(4.3,2.5)(5.1,1.9)(4.9,2.5)
\psline[linewidth=1.5pt]{->}(1,1.45)(1,0.65)
\psline[linewidth=1.5pt]{->}(0.6,0.7)(0.6,1.5)
\psline[linewidth=1.5pt]{->}(2.8,1.45)(2.8,0.65)
\psline[linewidth=1.5pt]{->}(2.4,0.7)(2.4,1.5)
\psline[linewidth=1.5pt]{->}(4.6,1.45)(4.6,0.65)
\psline[linewidth=1.5pt]{->}(4.2,0.7)(4.2,1.5)
\rput(1.2,1){$\pi$}
\rput(0.4,1){$\phi$}
\rput(3,1){$\pi$}
\rput(2.2,1){$\phi$}
\psline[linewidth=1.5pt]{->}(1.55,2)(2.15,2)
\psline[linewidth=1.5pt]{->}(3.35,2)(3.95,2)
\psline[linewidth=1.5pt]{->}(0.5,-0.75)(4.3,-0.75)
\rput(4.8,1){$\pi$}
\rput(4,1){$\phi$}
\rput(0.9,2.82){\small $K_\Phi(\theta_i)$}
\rput(2.7,2.82){\small $K_\Phi(\theta_j)$}
\rput(4.5,2.82){\small $K_\Phi(\theta_k)$}
\rput(1.8,2.3){\small $\psi_{ji}$}
\rput(3.6,2.3){\small $\psi_{kj}$}
\rput(4.8,-0.7){\small time}
\end{pspicture}
\label{subfig:area_persist}
}
\end{subfigure}
\vspace{-5mm}
\caption{Persistence over time in CW complexes.} 
\label{fig:persist}
\end{figure}

In Theorem~\ref{thm:gauge group}, $t_{ij}$ is the \emph{transition function} and $G$ is the \emph{structure group} or the \emph{gauge group}. Since $G$ is a group of transition functions $t_{ij}\in G$, $t_{ij}$ must satisfy certain conditions. That is, $t_{ii}=1$ is the identity element and $t_{ij}t_{ji}=1$ gives the inverse of each element. Moreover, there is a group operation such that $t_{ik}=t_{ij}t_{jk}$.  If there is no other $g\in G$ except the identity, such that $gx=x$ for all $f \in F$, then $G$ is a group of homeomorphisms on $F$.

\section{Application: Persistence in CW Spaces}.\\
We illustrate Theorem~\ref{thm:gauge group} and its implications, using an expansion of Fig.\ref{fig:deshom} shown in Fig.~\ref{fig:persist}. The probe function $\phi:2^K \rightarrow \mathbb{R}^2$, maps each region to a feature vector, $[temperature, area]$.   Consider, for example, two subsets $U_i, U_j\in 2^K$such that $U_i\cap U_j \neq \emptyset $.    $U_i, U_j$ are represented by two planar regions in Fig.~\ref{fig:persist}, namely, $U_i$ in gray and $U_j$ in light gray are two regions in the base space $K$, which have been colored for the sake of distinction. The intersection in this case is the shared $\sigma^1$.   We assume that the functions $\phi$ and $\pi$ change with respect to the value of a parameter, $\{\theta_i\}_i$.

The parameter $\{\theta_i\}_i$ can be thought of as time or spatial location.   In this example, we consider $\{\theta_i\}_i$ to be time.
We will consider the changes in each of the components of the feature vector, namely temperature (Fig.~\ref{subfig:temp_persist}) and area (Fig.~\ref{subfig:area_persist}) separately.
Using the local trivialization condition for $\theta_i$, the fibre for the region $U_i$, $\phi(U_i)=[green,0.25]$  and the fibre for region $U_j$ is $\phi(U_j)=[red,0.75]$.
The temperature is represented by color, where temperature $red > yellow >green$. 
Thus, Theorem~\ref{thm:gauge group} states that the transition of the description from the region $U_i$ to $U_j$ across the common intersection is governed by transition functions that form a group $G$.   Next, consider $U_i,U_j$ for different values of the parameters, {\em namely}, $\theta_j, \theta_k$. 

Let us first look at Fig.~\ref{subfig:temp_persist}, in which the areas remain unchanged while the temperature of $U_i$ decreases as it changes from red to yellow and then to green. 
In keeping with changing temperature, the transition functions that describe the changes across $U_i \cap U_j$ also change. This leads to a change in the group $G$ with the value of the temperature parameter.
Thus, a change in $G$ is an indicator of transitions in the description of the topological space with respect to the value of the parameters. 
Again, for example, consider Fig.~\ref{subfig:area_persist} in which the area of the $U_i$ reduces with time while the area of $U_j$ remains the the same. This also results in a change in the transition function values and hence in the group $G$.

We can combine these observations in the study of persistence with a shape signature introduced in \cite[\S 2.5]{Peters2017AMSJsignature}, to develop shape signatures for topological spaces based on the description. Persistence in CW spaces focuses on the stability of topological signatures with respect to one or more parameters. 


\bibliographystyle{amsplain}
\bibliography{NSrefs}

\end{document}